\documentclass[a4paper,12pt,reqno]{amsart}
\usepackage{amsmath,amsthm,amssymb}
\usepackage{graphicx, color}
\usepackage[numbers,sort&compress]{natbib}
\nonstopmode \numberwithin{equation}{section}

\newtheorem{theorem}{Theorem}[section]

 \newtheorem{corollary}{Corollary}[section]

\newtheorem{lemma}{Lemma}[section]

\allowdisplaybreaks

\allowdisplaybreaks

\begin{document}
\title[Pathway of $\mathtt{k}$-Struve function]{ Pathway fractional integral operators involving  $\mathtt{k}$-Struve function}

\author{Kottakkaran.S.Nisar}
\address{Department of Mathematics, College of Arts and Science-Wadi Al dawser,11991,\\
Prince Sattam bin Abdulaziz University, Saudi Arabia}
\email{ksnisar1@gmail.com, n.sooppy@psau.edu.sa}

\author{Saiful. R. Mondal}
\address{Department of Mathematics\\
King Faisal University, Al Ahsa 31982, Saudi Arabia}
\email{smondal@kfu.edu.sa}


\begin{abstract}
A new generalization called $\mathtt{k}$-Struve function and its properties given by Nisar and saiful very recently. In this paper, we establish the pathway fractional integral representation of  $\mathtt{k}$-Struve function. Many special cases also established to obtain the pathway integral representation of classical Struve function.  
\end{abstract}
\maketitle

\section{Introduction}
\label{Intro}

Recently, Nisar et al.\cite{Nisar-Saiful} introduced and studied various properties of $\mathtt{k}$-Struve function $\mathtt{S}_{\nu,c}^{\mathtt{k}}$ defined by

\begin{equation}\label{k-Struve}
\mathtt{S}_{\nu,c}^{\mathtt{k}}(x):=\sum_{r=0}^{\infty}\frac{(-c)^r}
{\Gamma_{\mathtt{k}}(r\mathtt{k}+\nu+\frac{3\mathtt{k}}{2})\Gamma(r+\frac{3}{2})}
\left(\frac{x}{2}\right)^{2r+\frac{\nu}{\mathtt{k}}+1}.
\end{equation}
where $c,\nu \in \mathbb{C}, \nu>\frac{3}{2}\mathtt{k}$.
The generalized Wright hypergeometric function ${}_{p}\psi _{q}(z)$ is given by
the series 
\begin{equation}
{}_{p}\psi _{q}(z)={}_{p}\psi _{q}\left[ 
\begin{array}{c}
(a_{i},\alpha _{i})_{1,p} \\ 
(b_{j},\beta _{j})_{1,q}%
\end{array}%
\bigg|z\right] =\displaystyle\sum_{k=0}^{\infty }\dfrac{\prod_{i=1}^{p}%
\Gamma (a_{i}+\alpha _{i}k)}{\prod_{j=1}^{q}\Gamma (b_{j}+\beta _{j}k)}%
\dfrac{z^{k}}{k!},  \label{Fox-Wright}
\end{equation}%
where $a_{i},b_{j}\in \mathbb{C}$, and real $\alpha _{i},\beta _{j}\in 
\mathbb{R}$ ($i=1,2,\ldots ,p;j=1,2,\ldots ,q$). Asymptotic behavior of this
function for large values of argument of $z\in {\mathbb{C}}$ were studied in 
\cite{CFox} and under the condition 
\begin{equation}
\displaystyle\sum_{j=1}^{q}\beta _{j}-\displaystyle\sum_{i=1}^{p}\alpha
_{i}>-1  \label{eqn-5-Struve}
\end{equation}%
was found in the work of \cite{Wright-2,Wright-3}. Properties of this
generalized Wright function were investigated in \cite{Kilbas}, (see also 
\cite{Kilbas-itsf, Kilbas-frac}. In particular, it was proved \cite{Kilbas}
that ${}_{p}\psi _{q}(z)$, $z\in {\mathbb{C}}$ is an entire function under
the condition ($\ref{eqn-5-Struve}$).

In\cite{Nair-1}, Nair introduced a pathway fractional integral operator and developed further by
Mathai and Haubold \cite{Mathai-Habold-1} ,\cite{Mathai-Habold-2} (Also, see \cite{Mathai-pathway}) is defined as follows :

Let $\ f\left( x\right) \in L\left( a,b\right) ,\eta \in \mathbb{C},\Re\left( \eta
\right) >0,a>0$ and the pathway parameter $\alpha <1$(cf \cite%
{Praveen-pathway}),then

\begin{equation}
\left( P_{0+}^{\left( \eta ,\alpha \right) }f\right) \left( x\right)
=x^{\eta }\int\limits_{0}^{\left[ \frac{x}{a\left( 1-\alpha \right)}\right]
}\left[1- \frac{a\left( 1-\alpha \right) t}{x}\right] ^{\frac{\eta }{\left(
1-\alpha \right)}}f\left( t\right) dt.  \label{eqn-path-1}
\end{equation}

For a real scalar $\alpha$, the pathway model for scalar random variables is
represented by the following probability density function (p.d.f.):

\begin{equation}
f\left( x\right) =c\left\vert x\right\vert ^{\gamma -1}\left[ 1-a\left(
1-\alpha \right) \left\vert x\right\vert ^{\delta }\right] ^{\frac{\beta }{%
\left( 1-\alpha \right) }},  \label{eqn-path-2}
\end{equation}

provided that $-\infty <x<\infty ,\delta >0,\beta \geq 0,\left[ 1-a\left(
1-\alpha \right) \left\vert x\right\vert ^{\delta }\right] >0,$ and $\gamma
>0$, where is the normalizing constant and $\alpha$ is called the pathway parameter \cite{Nair-1}.
 
%\begin{eqnarray*}
%c &=&\frac{1}{2}\frac{\delta \left[ a\left( 1-\alpha \right) \right] ^{\frac{%
%\gamma }{\delta }}\Gamma \left( \frac{\gamma }{\delta }+\frac{\beta }{%
%1-\alpha }+1\right) }{\Gamma \left( \frac{\gamma }{\delta }\right) \Gamma
%\left( \frac{\beta }{1-\alpha }+1\right) },\text{for }\alpha <1 \\
%&=&\frac{1}{2}\frac{\delta \left[ a\left( 1-\alpha \right) \right] ^{\frac{%
%\gamma }{\delta }}\Gamma \left( \frac{\beta }{\alpha -1}\right) }{\Gamma
%\left( \frac{\gamma }{\delta }\right) \Gamma \left( \frac{\beta }{\alpha -1}-%
%\frac{\gamma }{\delta }\right) },\text{for }\frac{1}{1-\alpha }-\frac{\gamma 
%}{\delta }>0,\alpha >1 \\
%&=&\frac{1}{2}\frac{\left( a\beta \right) ^{\frac{\gamma }{\delta }}}{\Gamma
%\left( \frac{\gamma }{\delta }\right) },\alpha \rightarrow 1
%\end{eqnarray*}

Note that for $\alpha <1$ it is a finite range density with $\left[
1-a\left( 1-\alpha \right) \left\vert x\right\vert ^{\delta }\right] >0$ and
\ ($\ref{eqn-path-2}$) remains in the extended generalized type-1 beta
family \ . The pathway density in ($\ref{eqn-path-2}$), for $\alpha < 1$,
includes the extended type-1 beta density, the triangular density, the
uniform density and many other p.d.f'.s .\cite{Praveen-pathway}.For instance
, $\alpha >1$ gives
\begin{equation}
f\left( x\right) =c\left\vert x\right\vert ^{\gamma -1}\left[ 1+a\left(
1-\alpha \right) \left\vert x\right\vert ^{\delta }\right] ^{-\frac{\beta }{%
\left( 1-\alpha \right) }},  \label{eqn-path-3}
\end{equation}%

provided that $-\infty <x<\infty ,\delta >0,\beta \geq 0,$ and $\alpha >0$
which is the extended generalized type-2 beta model for real x. It includes
the type-2 beta density, the F density, the Student-t density, the Cauchy
density and many more. For more details about pathway integral operator, one can refer \cite{Praveen-pathway, Purohit}.The purpose of this work is to investigate the composition formula of integral transform operator due to Nair, which is expressed in terms of the generalized Wright hypergeometric function, by inserting the $\mathtt{k}-$ Struve function 

\section{Pathway Fractional Integration of$\mathtt{k}$-Struve function.}

The results given in this section are based on the preliminary assertions
giving by composition formula of pathway fractional integral ($\ref%
{eqn-path-1}$) with a power function.

\begin{lemma}
({Agarwal}~\cite{Praveen-pathway},Lemma 1)

Let $\eta \in \mathbb{C},\Re\left( \eta \right) >0,\beta \in \mathbb{C}$ and $\alpha <1.$ If $%
\Re\left( \beta \right) >0,$and $\Re\left( \frac{\eta }{1-\alpha }\right) >-1,$%
then

\begin{equation}
\left\{ P_{0+}^{\left( \eta ,\alpha \right) }\left[ t^{\beta -1}\right]
\right\} \left( x\right) =\frac{x^{\eta +\beta }}{\left[ a\left( 1-\alpha
\right) \right] ^{\beta }}\frac{\Gamma \left( \beta \right) \Gamma \left( 1+%
\frac{\eta }{1-\alpha }\right) }{\Gamma \left( 1+\frac{\eta }{1-\alpha }+\beta \right) }.
\label{lemma1}
\end{equation}
The pathway fractional integration of the $\mathtt{k}-$ Struve function is given by the following theorem.
\end{lemma}

\begin{theorem}\label{Th1}
Let $\eta ,\rho ,\nu, c \in C$ and $\alpha <1$ be such that $\Re\left(
\eta \right) >0,\Re\left( \rho \right) >0, \nu>-\frac{3}{2}\mathtt{k}$ and $\Re\left( \frac{\eta }{1-\alpha 
}\right) >-1 $ then the following formula hold true%

\begin{equation}\label{eqn1-th1}
\begin{array}{c}
P_{0+}^{\left( \eta ,\alpha \right) }\left[ t^{\rho-1}\mathtt{S}_{\nu,c}^{\mathtt{k}}(t)\right]
\left( x\right) =x^{\eta}\left(\frac{x}{a(1-\alpha)}\right)^{\rho+\frac{\nu}{\mathtt{k}}+1}\frac{\Gamma\left(1+\frac{\eta}{1-\alpha}\right)}
{\mathtt{k}^{\frac{\nu}{\mathtt{k}}+\frac{1}{2}}2^{\frac{\nu}{\mathtt{k}}+1}}\\ 
\times _{2}\Psi _{3}\left[ 
\begin{array}{ccc}
\left( \rho +\frac{\nu}{\mathtt{k}}+1,2\right) , & \left( 1,1\right); &  \\ 
\left( \rho +\frac{\nu}{\mathtt{k}}+\frac{\eta }{1-\alpha }+2,2\right) , & \left(\frac{\nu}{\mathtt{k}}+\frac{3}{2},1\right) , & \left( 3/2,1\right)%
\end{array}%
;-\frac{cx^{2}}{4\mathtt{k}\left[ a^{2}\left( 1-\alpha \right)^{2} \right]}\right].%
\end{array}
\end{equation}
\end{theorem}

\begin{proof}
Applying the pathway operator defined in \eqref{eqn-path-1} to \eqref{k-Struve}, and changing the order of integration and summation, we get%

\begin{align*}
\left( P_{0+}^{\left( \eta ,\alpha \right)}\left[ t^{\rho-1}\mathtt{S}_{\nu,c}^{\mathtt{k}}(t)\right] \right)
\left( x\right)&=P_{0+}^{\left( \eta ,\alpha \right)}\left[t^{\rho-1}\sum_{r=0}^{\infty}\frac{(-c)^{r}\left(\frac{t}{2}\right)^{2r+\frac{\nu}{\mathtt{k}}+1}}{\Gamma_{\mathtt{k}}\left(r\mathtt{k}+\nu+\frac{3}{2}\mathtt{k}\right)\Gamma\left(r+\frac{3}{2}\right)}\right](x)\\
&=\sum_{r=0}^{\infty}\frac{(-c)^{r}\left(\frac{1}{2}\right)^{2r+\frac{\nu}{\mathtt{k}}+1}}
{\Gamma_{\mathtt{k}}\left(r\mathtt{k}+\nu+\frac{3}{2}\mathtt{k}\right)\Gamma\left(r+\frac{3}{2}\right)}
P_{0+}^{\left( \eta ,\alpha \right)}\left(t^{\rho+2r+\frac{\nu}{\mathtt{k}}}\right)(x)
\end{align*}
Using Lemma $(\ref{lemma1})$, we get
\begin{align*}
&&=\sum_{r=0}^{\infty}\frac{(-c)^{r}\left(\frac{1}{2}\right)^{2r+\frac{\nu}{\mathtt{k}}+1}}
{\Gamma_{\mathtt{k}}\left(r\mathtt{k}+\nu+\frac{3}{2}\mathtt{k}\right)\Gamma\left(r+\frac{3}{2}\right)}
\frac{x^{\eta+\rho+2r+\frac{\nu}{\mathtt{k}}+1}}{[a(1-\alpha)]^{\rho+2r+\frac{\nu}{\mathtt{k}}+1}}\\
&&\times\frac{\Gamma\left(\rho+2r+\frac{\nu}{\mathtt{k}}+1\right)\Gamma\left(1+\frac{\eta}{1-\alpha}\right)}{\Gamma\left(\frac{\eta}{1-\alpha}+\rho+2r+\frac{\nu}{\mathtt{k}}+2\right)}
\end{align*}
Now using the relation $\Gamma_{\mathtt{k}}\left(\gamma\right)=\mathtt{k}^{\frac{\gamma}{\mathtt{k}}-1}\Gamma\left(\frac{\gamma}{\mathtt{k}}\right)$, we get
\begin{align*}
&=\frac{x^{\eta+\rho+\frac{\nu}{\mathtt{k}}+1}\Gamma\left(1+\frac{\eta}{1-\alpha}\right)}{\left[a(1-\alpha)\right]^{\rho+\frac{\nu}{\mathtt{k}}+1}2^{\frac{\nu}{\mathtt{k}}+1}}\\
&\times \sum_{r=0}^{\infty}\frac{(-c)^{r}x^{2r}}
{\mathtt{k}^{r+\frac{\nu}{\mathtt{k}}+\frac{1}{2}}\Gamma\left(r+\frac{\nu}{\mathtt{k}}+\frac{3}{2}\right)\Gamma\left(r+\frac{3}{2}\right)4^{r}[a(1-\alpha)]^{2r}}\\
&\times\frac{\Gamma\left(\rho+\frac{\nu}{\mathtt{k}}+1+2r\right)}{\Gamma(\frac{\eta}{1-\alpha}+\rho+2r+\frac{\nu}{\mathtt{k}}+2)}.
\end{align*}
In view of $(\ref{Fox-Wright})$, we arrived the desired result.
\end{proof}
\begin{corollary}
If we take $\mathtt{k}=1$ in theorem \eqref{Th1}, then we get the pathway integrals involving classical Struve function as:
\begin{equation}\label{eqn1-cor1}
\begin{array}{c}
P_{0+}^{\left( \eta ,\alpha \right) }\left[ t^{\rho-1}\mathtt{S}_{\nu,c}^{1}(t)\right]
\left( x\right) =x^{\eta}\left(\frac{x}{a(1-\alpha)}\right)^{\rho+\nu+1}
\frac{\Gamma\left(1+\frac{\eta}{1-\alpha}\right)}{2^{\nu+1}}\\ 
\times _{2}\Psi _{3}\left[ 
\begin{array}{ccc}
\left( \rho +\nu+1,2\right) , & \left( 1,1\right); &  \\ 
\left( \rho +\nu+\frac{\eta }{1-\alpha }+2,2\right) , & \left(\nu+\frac{3}{2},1\right) , & \left( 3/2,1\right)%
\end{array}%
;-\frac{cx^{2}}{4\left[ a^{2}\left( 1-\alpha \right)^{2} \right]}\right].%
\end{array}
\end{equation}
\end{corollary}

Now, we will give the relation relation between trigonometric function and $\mathtt{k}$-Struve function. By taking $\nu=\mathtt{k}/2$ 
in (3.10) of \cite{Nisar-Saiful} we get the relation between cosine functions and $\mathtt{k}$-Struve functions as
\begin{equation}\label{cos}
1-\cos\left(\frac{\alpha x}{\sqrt{\mathtt{k}}}\right)= \frac{\alpha}{\mathtt{k}}\sqrt{\frac{\pi x}{2}}~\mathtt{S}_{\frac{\mathtt{k}}{2}, \alpha^2}^{\mathtt{k}} (x).
\end{equation}

Similarly, the relation

\begin{equation}\label{cosh}
\cosh\left(\frac{\alpha x}{\sqrt{\mathtt{k}}}\right)-1= \frac{\alpha}{\mathtt{k}}\sqrt{\frac{\pi x}{2}}~\mathtt{S}_{\frac{\mathtt{k}}{2}, -\alpha^2}^{\mathtt{k}} (x),
\end{equation}
can be derive from (3.11) of \cite{Nisar-Saiful}.

Also, by taking $\nu=-\frac{k}{2}$ in $(\ref{k-Struve})$, we obtained the following:
\begin{equation}\label{sin}
\sin\left(\frac{\alpha x}{\sqrt{\mathtt{k}}}\right)=\alpha\left(\sqrt{\frac{\pi x}{2\mathtt{k}}}\right)\mathtt{S}_{-\frac{\mathtt{k}}{2},\alpha^{2}}^{\mathtt{k}}\left(x\right).
\end{equation}

\begin{equation}\label{sinh}
\sinh\left(\frac{\alpha x}{\sqrt{\mathtt{k}}}\right)=\alpha\left(\sqrt{\frac{\pi x}{2\mathtt{k}}}\right)\mathtt{S}_{-\frac{\mathtt{k}}{2},-\alpha^{2}}^{\mathtt{k}}\left(x\right).
\end{equation}

\section{Pathway fractional integration of cosine,hyperbolic cosine, sine
and hyperbolic sine functions}
\begin{theorem}\label{Th2}
Let $\eta ,\rho ,\nu, c \in \mathbb{C}$ and $\alpha <1$ be such that $\Re\left(
\eta \right) >0,\Re\left( \rho \right) >0, \nu>-\frac{3}{2}\mathtt{k}$ and $\Re\left( \frac{\eta }{1-\alpha 
}\right) >-1 $ then the following formula hold true%

\begin{equation}\label{eqn1-th2}
\begin{array}{c}
P_{0+}^{\left( \eta ,\alpha \right) }\left[ t^{\rho-1}\left(1-\cos\left(\frac{\gamma t}{\sqrt{\mathtt{k}}}\right)\right)\right]
\left( x\right) =\sqrt{\pi}\frac{\gamma}{\mathtt{k}^{2}}\frac{ x^{\eta+\rho+2}}{4[a(1-\alpha)]^{\rho+2}}\Gamma\left(1+\frac{\eta}{1-\alpha}\right)\\
\times _{2}\Psi _{3}\left[ 
\begin{array}{ccc}
\left( \rho +2,2\right) , & \left( 1,1\right); &  \\ 
\left( \rho +3+\frac{\eta }{1-\alpha },2\right) , & \left(2,1\right) , & \left( 3/2,1\right)%
\end{array}%
;-\frac{\gamma^{2}x^{2}}{4\mathtt{k}\left[ a^{2}\left( 1-\alpha \right)^{2} \right]}\right]%
\end{array}
\end{equation}
\end{theorem}

\begin{proof}
Applying the pathway operator defined in \eqref{eqn-path-1} to \eqref{cos}, and changing the order of integration and summation, we get%
\begin{align*}
P_{0+}^{\left( \eta ,\alpha \right) }\left[ t^{\rho-1}\left(1-\cos\left(\frac{\gamma t}{\sqrt{\mathtt{k}}}\right)\right)\right]
&=\left( P_{0+}^{\left( \eta ,\alpha \right)}\left[ t^{\rho-1}\frac{\gamma}{\mathtt{k}}\sqrt{\frac{\pi t}{2}}\mathtt{S}_{\frac{k}{2},\gamma^{2}}^{\mathtt{k}}(t)\right] \right)\left( x\right)\\
&=\sqrt{\frac{\pi}{2}}\frac{\gamma}{\mathtt{k}}\sum_{r=0}^{\infty}\frac{(-\gamma^{2})^{r}\left(\frac{1}{2}\right)^{2r+\frac{1}{2}+1}}{\Gamma_{\mathtt{k}}\left(r\mathtt{k}+\frac{4}{2}\mathtt{k}\right)\Gamma\left(r+\frac{3}{2}\right)}P_{0+}^{\left( \eta ,\alpha \right)}\left[t^{\rho+2r+1}\right](x),
\end{align*}
Using Lemma \ref{lemma1}, we get
\begin{align*}
=\sqrt{\pi}\frac{\gamma}{\mathtt{k}}\sum_{r=0}^{\infty}\frac{(-\gamma^{2})^{r}(\frac{1}{2})^{2r+2}}{\Gamma_{\mathtt{k}}(r\mathtt{k}+2\mathtt{k})\Gamma(r+\frac{3}{2})}\frac{x^{\eta+\rho+2+2r}}{[a(1-\alpha)]^{\rho+2+2r}}\\
\times \frac{\Gamma(\rho+2r+2)\Gamma(1+\frac{\eta}{1-\alpha})}{\Gamma(1+\frac{\eta}{1-\alpha}+\rho+2+2r)}
\end{align*}
Now using the relation $\Gamma_{\mathtt{k}}\left(\gamma\right)=\mathtt{k}^{\frac{\gamma}{\mathtt{k}}-1}\Gamma\left(\frac{\gamma}{\mathtt{k}}\right)$, we get
\begin{align*}
&=\sqrt{\pi}\frac{\gamma}{\mathtt{k}^{2}}\frac{ x^{\eta+\rho+2}}{4[a(1-\alpha)]^{\rho+2}}\Gamma\left(1+\frac{\eta}{1-\alpha}\right)\\
&=\sum_{r=0}^{\infty}\frac{(-\gamma^{2})^{r}(\frac{1}{2})^{2r}\Gamma\left(\rho+2+2r\right)}{\Gamma(r+2)\Gamma(r+\frac{3}{2})}\frac{x^{2r}}{[a(1-\alpha)]^{2r}\Gamma(1+\frac{\eta}{1-\alpha}+\rho+2+2r)}
\end{align*}
In view of $(\ref{Fox-Wright})$, we arrived the desired result.
\end{proof}

\begin{corollary}\label{Cor2}
If we take $\mathtt{k}=1$ in theorem \ref{Th2}, then we get the pathway integrals involving classical Struve function as:
Let $\eta ,\rho ,\nu, c \in \mathbb{C}$ and $\alpha <1$ be such that $\Re\left(
\eta \right) >0,\Re\left( \rho \right) >0, \nu>-\frac{3}{2}$ and $\Re\left( \frac{\eta }{1-\alpha 
}\right) >-1 $ then the following formula hold%
\begin{equation}\label{eqn1-Cor2}
\begin{array}{c}
P_{0+}^{\left( \eta ,\alpha \right)}\left[t^{\rho-1}\left(1-cos{\gamma t}\right)\right]\left( x\right) =\sqrt{\pi}{\gamma}\frac{x^{\eta+\rho+2}}{4[a(1-\alpha)]^{\rho+2}}\Gamma\left(1+\frac{\eta}{1-\alpha}\right)\\
\times _{2}\Psi _{3}\left[ 
\begin{array}{ccc}
\left( \rho +2,2\right) , & \left( 1,1\right); &  \\ 
\left( \rho +3+\frac{\eta }{1-\alpha },2\right) , & \left(2,1\right) , & \left( 3/2,1\right)%
\end{array}%
;-\frac{\gamma^{2}x^{2}}{4\left[ a^{2}\left( 1-\alpha \right)^{2} \right]}\right]%
\end{array}
\end{equation}
\end{corollary}

\begin{theorem}\label{Th3}
Let $\eta ,\rho ,\nu, c \in \mathbb{C}$ and $\alpha <1$ be such that $\Re\left(
\eta \right) >0,\Re\left( \rho \right) >0, \nu>-\frac{3}{2}\mathtt{k}$ and $\Re\left( \frac{\eta }{1-\alpha 
}\right) >-1 $ then the following formula hold true%
\begin{equation}\label{eqn1-th3}
\begin{array}{c}
P_{0+}^{\left( \eta ,\alpha \right) }\left[ t^{\rho-1}\left(\cosh\left(\frac{\gamma t}{\sqrt{\mathtt{k}}}\right)\right)\right]
\left( x\right) =\sqrt{\pi}\frac{\gamma}{\mathtt{k}^{2}}\frac{ x^{\eta+\rho+2}}{4[a(1-\alpha)]^{\rho+2}}\Gamma\left(1+\frac{\eta}{1-\alpha}\right)\\
\times _{2}\Psi _{3}\left[ 
\begin{array}{ccc}
\left( \rho +2,2\right) , & \left( 1,1\right); &  \\ 
\left( \rho +3+\frac{\eta }{1-\alpha },2\right) , & \left(2,1\right) , & \left( 3/2,1\right)%
\end{array}%
;\frac{\gamma^{2}x^{2}}{4\mathtt{k}\left[ a^{2}\left( 1-\alpha \right)^{2} \right]}\right]%
\end{array}
\end{equation}
\end{theorem}

\begin{corollary}\label{cor3}
If we set $\mathtt{k}=1$ in Theorem \ref{Th3} then we get,
Let $\eta ,\rho ,\nu, c \in \mathbb{C}$ and $\alpha <1$ be such that $\Re\left(
\eta \right) >0,\Re\left( \rho \right) >0, \nu>-\frac{3}{2}$ and $\Re\left( \frac{\eta }{1-\alpha 
}\right) >-1 $ then the following formula hold true%
\begin{equation}\label{eqn1-Cor3}
\begin{array}{c}
P_{0+}^{\left( \eta ,\alpha \right) }\left[ t^{\rho-1}\left(\cosh\left(\gamma t\right)\right)\right]
\left( x\right) =\sqrt{\pi}\gamma\frac{ x^{\eta+\rho+2}}{4[a(1-\alpha)]^{\rho+2}}\Gamma\left(1+\frac{\eta}{1-\alpha}\right)\\
\times _{2}\Psi _{3}\left[ 
\begin{array}{ccc}
\left( \rho +2,2\right) , & \left( 1,1\right); &  \\ 
\left( \rho +3+\frac{\eta }{1-\alpha },2\right) , & \left(2,1\right) , & \left( 3/2,1\right)%
\end{array}%
;\frac{\gamma^{2}x^{2}}{4\mathtt{k}\left[ a^{2}\left( 1-\alpha \right)^{2} \right]}\right]%
\end{array}
\end{equation}
\end{corollary}

\begin{theorem}\label{Th4}
Let $\eta ,\rho ,\nu, c \in \mathbb{C}$ and $\alpha <1$ be such that $\Re\left(
\eta \right) >0,\Re\left( \rho \right) >0, \nu>-\frac{3}{2}\mathtt{k}$ and $\Re\left( \frac{\eta }{1-\alpha 
}\right) >-1 $ then the following formula hold true%
\begin{equation}\label{eqn1-th4}
\begin{array}{c}
P_{0+}^{\left( \eta ,\alpha \right) }\left[ t^{\rho-1}\left(\sin\left(\frac{\gamma t}{\sqrt{\mathtt{k}}}\right)\right)\right]
\left( x\right) =\gamma\sqrt{\frac{\pi}{\mathtt{k}}}\frac{x^{\rho+\eta+1}}{2}{\left[a(1-\alpha)\right]^{\rho+\frac{1}{2}}}\Gamma\left(1+\frac{\eta}{1-\alpha}\right)\\
\times _{1}\Psi _{2}\left[ 
\begin{array}{ccc}
\left( \rho +\frac{1}{2},2\right) ; &  \\ 
\left( \rho +\frac{\eta }{1-\alpha }+\frac{3}{2},2\right) , & \left(\frac{3}{2},1\right); & %
\end{array}%
;\frac{-\gamma^{2}x^{2}}{4\mathtt{k}\left[ a^{2}\left( 1-\alpha \right)^{2} \right]}\right]%
\end{array}
\end{equation}
\end{theorem}

\begin{proof}
Applying the pathway operator defined in \eqref{eqn-path-1} to \eqref{cos}, and changing the order of integration and summation, we get%
\begin{align*}
P_{0+}^{\left( \eta ,\alpha \right) }\left[ t^{\rho-1}\left(\sin\left(\frac{\gamma t}{\sqrt{\mathtt{k}}}\right)\right)\right]
&=\left( P_{0+}^{\left( \eta ,\alpha \right)}\left[ t^{\rho-1}\gamma\sqrt{\frac{\pi x}{2\mathtt{k}}}~\mathtt{S}_{-\frac{\mathtt{k}}{2},\gamma^{2}}^{\mathtt{k}}(t)\right] \right)\left( x\right)\\
&=\gamma\sqrt{\frac{\pi x}{2\mathtt{k}}}P_{0+}^{\left( \eta ,\alpha \right)}\left[t^{\rho-1}\sum_{r=0}^{\infty}\frac{(-\gamma^{2})^{r}
\left(\frac{t}{2}\right)^{2r-\frac{1}{2}+1}}{\Gamma_{\mathtt{k}}(r\mathtt{k}-\frac{\mathtt{k}}{2}+\frac{3\mathtt{k}}{2})\Gamma\left(r+\frac{3}{2}\right)}\right](x)\\
&=\gamma\sqrt{\frac{\pi x}{2\mathtt{k}}}\sum_{r=0}^{\infty}\frac{(-\gamma^{2})^{r}\left(\frac{1}{2}\right)^{2r+\frac{1}{2}}}{\Gamma_{\mathtt{k}}\left(r\mathtt{k}+\mathtt{k}\right)\Gamma\left(r+\frac{3}{2}\right)}P_{0+}^{\left( \eta ,\alpha \right)}\left[t^{\rho+2r+\frac{1}{2}-1}\right]
\end{align*}
Using Lemma \ref{lemma1} and the relation $\Gamma_{\mathtt{k}}\left(\gamma\right)=\mathtt{k}^{\frac{\gamma}{\mathtt{k}}-1}\Gamma\left(\frac{\gamma}{\mathtt{k}}\right)$, we get
\begin{align*}
&=\gamma\sqrt{\frac{\pi}{\mathtt{k}}}\frac{x^{\rho+\eta+1}}{[a(1-\alpha)]^{\rho+\frac{1}{2}}}\\
&\times \sum_{r=0}^{\infty}\frac{(-\gamma^{2})^{r}\left(\frac{1}{2}\right)^{2r}x^{2r}\Gamma(\rho+\frac{1}{2}+2r)}{\Gamma\left(r+1\right)\Gamma\left(r+\frac{3}{2}\right)\Gamma\left(\rho+\frac{\eta}{1-\alpha}+\frac{1}{2}+1+2r\right)k^{r}[a(1-\alpha)]^{2r}}
\end{align*}
In view of $(\ref{Fox-Wright})$, we arrived the desired result.
\end{proof}

\begin{corollary}\label{Cor4}
If we take $\mathtt{k}=1$, then we have
Let $\eta ,\rho ,\nu, c \in \mathbb{C}$ and $\alpha <1$ be such that $\Re\left(
\eta \right) >0,\Re\left( \rho \right) >0, \nu>-\frac{3}{2}$ and $\Re\left( \frac{\eta }{1-\alpha 
}\right) >-1 $ then the following formula hold true%
\begin{equation}\label{eqn1-Cor4}
\begin{array}{c}
P_{0+}^{\left( \eta ,\alpha \right) }\left[ t^{\rho-1}\left(\sin\left(\gamma t\right)\right)\right]
\left( x\right) =\gamma\sqrt{\pi}\frac{x^{\rho+\eta+1}}{2}{\left[a(1-\alpha)\right]^{\rho+\frac{1}{2}}}\Gamma\left(1+\frac{\eta}{1-\alpha}\right)\\
\times _{1}\Psi _{2}\left[ 
\begin{array}{ccc}
\left( \rho +\frac{1}{2},2\right) ; &  \\ 
\left( \rho +\frac{\eta }{1-\alpha }+\frac{3}{2},2\right) , & \left(\frac{3}{2},1\right); & %
\end{array}%
;\frac{-\gamma^{2}x^{2}}{4\left[ a^{2}\left( 1-\alpha \right)^{2} \right]}\right]%
\end{array}
\end{equation}
\end{corollary}

\begin{theorem}\label{Th5}
Let $\eta ,\rho ,\nu, c \in \mathbb{C}$ and $\alpha <1$ be such that $\Re\left(
\eta \right) >0,\Re\left( \rho \right) >0, \nu>-\frac{3}{2}\mathtt{k}$ and $\Re\left( \frac{\eta }{1-\alpha 
}\right) >-1 $ then the following formula hold true%
\begin{equation}\label{eqn1-th5}
\begin{array}{c}
P_{0+}^{\left( \eta ,\alpha \right) }\left[ t^{\rho-1}\left(\sinh\left(\frac{\gamma t}{\sqrt{\mathtt{k}}}\right)\right)\right]
\left( x\right) =\gamma\sqrt{\frac{\pi}{\mathtt{k}}}\frac{x^{\rho+\eta+1}}{2}{\left[a(1-\alpha)\right]^{\rho+\frac{1}{2}}}\Gamma\left(1+\frac{\eta}{1-\alpha}\right)\\
\times _{1}\Psi _{2}\left[ 
\begin{array}{ccc}
\left( \rho +\frac{1}{2},2\right) ; &  \\ 
\left( \rho +\frac{\eta }{1-\alpha }+\frac{3}{2},2\right) , & \left(\frac{3}{2},1\right); & %
\end{array}%
;\frac{\gamma^{2}x^{2}}{4\mathtt{k}\left[ a^{2}\left( 1-\alpha \right)^{2} \right]}\right]%
\end{array}
\end{equation}
\end{theorem}

\begin{corollary}\label{Cor5}
If we take $\mathtt{k}=1$, then we have
Let $\eta ,\rho ,\nu, c \in \mathbb{C}$ and $\alpha <1$ be such that $\Re\left(
\eta \right) >0,\Re\left( \rho \right) >0, \nu>-\frac{3}{2}$ and $\Re\left( \frac{\eta }{1-\alpha 
}\right) >-1 $ then the following formula hold true%
\begin{equation}\label{eqn1-Cor5}
\begin{array}{c}
P_{0+}^{\left( \eta ,\alpha \right) }\left[ t^{\rho-1}\left(\sinh\left(\gamma t\right)\right)\right]
\left( x\right) =\gamma\sqrt{\pi}\frac{x^{\rho+\eta+1}}{2}{\left[a(1-\alpha)\right]^{\rho+\frac{1}{2}}}\Gamma\left(1+\frac{\eta}{1-\alpha}\right)\\
\times _{1}\Psi _{2}\left[ 
\begin{array}{ccc}
\left( \rho +\frac{1}{2},2\right) ; &  \\ 
\left( \rho +\frac{\eta }{1-\alpha }+\frac{3}{2},2\right) , & \left(\frac{3}{2},1\right); & %
\end{array}%
;\frac{\gamma^{2}x^{2}}{4\left[ a^{2}\left( 1-\alpha \right)^{2} \right]}\right]%
\end{array}
\end{equation}
\end{corollary}

\end{document}